\documentclass[11pt]{amsart}
\usepackage[english]{babel}
\usepackage{graphicx}
\usepackage{pgf,tikz}
\usepackage{wrapfig}
\usepackage{pdfsync}
\usepackage{epsfig}
\usepackage{epstopdf}
\usepackage{latexsym,amssymb,amsfonts,amsmath, amscd, euscript, MnSymbol}
\usepackage{mathrsfs}
\usepackage{enumerate}
\usepackage{verbatim}
\usepackage{float}

\def\Power #1 { \powerset(#1) }
\def\Bidom #1 { {\mathfrak P} (#1) }

\newtheorem{definition}{{\bf Definition}}[section]
\newtheorem{theo}[definition]{{\bf Theorem}}

\newtheorem{cor}[definition]{{\bf Corollary}}

\newtheorem{proposition}[definition]{\noindent {\bf Proposition}}
\newtheorem{lem}[definition]{\noindent {\bf Lemma}}

\newtheorem{conjecture}[definition]{\noindent {\bf Conjecture}}

\newtheorem{problem}[definition]{\noindent {\bf Problem}}

\def\proofref #1 {{\noindent  {\bf Proof} (#1).}\ }

\def\endproof{\hfill {\kern 6pt\penalty 500
\raise -0pt\hbox{\vrule \vbox to5pt {\hrule width 5pt
\vfill\hrule}\vrule}}}
\def\centerpicture #1 by #2 (#3){\leavevmode
        \vbox to #2{
        \hrule width #1 height 0pt depth 0pt
        \vfill
        \special{pictfile #3}}}
\title[MacNeille completion of the free monoid]{A syntactic  approach to the MacNeille completion of $\Lambda^{\ast}$, the free monoid over an ordered alphabet $\Lambda$.}

\author[H-J.~Bandelt]{Hans-J\"urgen Bandelt
} 
\address{Fachbereich Mathematik, Universit\"at Hamburg, 
Bundesstr. 55, D-20146 Hamburg, Germany.}
\email{bandelt@math.uni-hamburg.de} 

\author[M.~Pouzet]{Maurice Pouzet} \address{Univ. Lyon, Universit\'e Claude Bernard Lyon 1, CNRS UMR 5208, Institut Camille Jordan, 43 blvd. du 11 novembre 1918, F-69622 Villeurbanne cedex, France,  and Mathematics \& Statistics Department, University of Calgary, Calgary, Alberta, Canada T2N 1N4}
 \email{pouzet@univ-lyon1.fr }


\date{\today}

\begin{document}

\begin{abstract} 
Let $\Lambda^{\ast}$ be the free monoid of (finite) words over a not 
necessarily finite alphabet $\Lambda$, which is equipped with some (partial) order. 
This ordering 
lifts to $\Lambda^{\ast}$, where it  extends the 
divisibility ordering of words. The MacNeille completion of 
$\Lambda^{\ast}$  constitutes a complete lattice ordered 
monoid and is  realized by the system of  "closed"  lower sets in $\Lambda^*$ (ordered by inclusion) or its isomorphic copy formed of the "closed" upper sets (ordered by  reverse inclusion). Under some additional hypothesis on $\Lambda$,  one can 
easily identify the closed lower sets as the finitely generated ones, whereas it is more complicated to determine the closed upper sets. For a fairly large class of ordered sets $\Lambda$ (including complete lattices as well as antichains) one can generate the closure of any upper set of words  by means of binary operations ( "syntactic rules") thus obtaining an efficient procedure to test closedness. Closed upper set of words are involved in an embedding theorem for valuated oriented graphs. In fact, generalized paths (so-called "zigzags") are encoded by words over an alphabet $\Lambda$. Then the valuated oriented graphs which are "isometrically" embeddable in a product  of zigzags have the characteristic property that the words corresponding to the zigzags between any pair of vertices form a closed upper set in $\Lambda^*$. 
\end{abstract}

\maketitle

\section{Introduction}
Our motivation for studying the MacNeille completion
stems from distance-preserving embeddings of graphs into products of 
path-like graphs. Here is an outline  of the embedding question. 
For undirected graphs there is a universal embedding theorem, due independently to Quilliot \cite{Qu1, Qu2} and Nowakowski and  Rival \cite {NoRi}:

{\it  Every undirected graph $G$ isometrically 
embeds in a product of paths}. 

"Isometry" and "product" certainly need 
a word of explanation. Isometry requires preservation of the shortest 
path distance, that is, any two vertices of $G$ are sent to two 
vertices at the same distance in the product. The product in question 
is the \emph{strong  product}, which is the canonical product in the 
category of undirected graphs with loops; namely, two vertices in 
such a product are adjacent if and only if all pairs of their 
corresponding coordinates form edges (which may be loops). Now, for 
directed graphs  (binary relations with loops) the embedding question is more intricate. The strong product one 
considers here is, of course, the direct product for reflexive 
relations,  but the potential factors are not just those directed graphs whose 
symmetric closures constitute undirected paths (see Kabil and Pouzet \cite{KaPo}). Distance and thus 
isometry have a natural meaning here, too -- but one needs to measure 
distances by sets of words  over a two or three lettter alphabet 
rather than numbers. The approriate notion of distance for directed 
graphs was introduced by Quilliot \cite{Qu1} and is subsumed in the general 
approach taken by Jawhari, Misane and  Pouzet \cite {JaMiPo} and  further developped by Pouzet and Rosenberg \cite{PR}. Let us focus on 
the case of oriented graphs (with loops), i.e., reflexive 
antisymmetric binary relations as in this case a two letter alphabet (distinguishing 'forward' and 'backward') will do: an {\it oriented\ graph}\ is a 
directed graph in which every pair of vertices is linked by at most 
one arc. The oriented analogues of undirected complete graphs, for 
instance, are the tournaments. The oriented versions of undirected 
paths are called {\it zigzags}, see Figure 1.

\vspace{1cm}
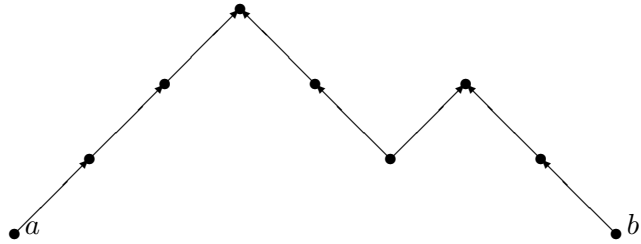
\begin{figure}[h]
\centering
\unitlength=1cm
\begin{picture}(16,4)
\put(1, 1){\circle*{.15} $a$}
\put(2,2){\circle*{.15}}
\put(3,3){\circle*{.15}}
\put(4,4){\circle*{.15}}
\put(5,3){\circle*{.15}}
\put(6,2){\circle*{.15}}
\put(7,3){\circle*{.15}}
\put(8,2){\circle*{.15}}
\put(9,1){\circle*{.15} $b$}
\put(1, 1){\vector (1,1){1}}
\put(2,2){\vector(1,1){1}}
\put(3,3){\vector(1,1){1}}
\put(5,3){\vector(-1,1){1}}
\put(6,2){\vector(-1,1){1}}
\put(6,2){\vector(1,1){1}}
\put(8,2){\vector(-1,1){1}}
\put(9,1){\vector(-1,1){1}}


\end{picture}

\caption{A zigzag from $a$ to $b$}

\end{figure}
The distance from the initial to the terminal vertex of a zigzag is 
not just a number (counting the edges) as in the undirected case but 
rather the isomorphy class of its homomorphic zigzag pre-images, thus 
a subset of what we call the {\it ordered\ monoid\ of\ zigzags}; 
this set consists of all zigzags (up to isomorphism) ordered as 
follows: $P \leq Q$ if and only if there is an arc-preserving 
mapping from the zigzag $Q$ onto the zigzag $P$ (which may collapse 
vertices because of the ubiquity of loops). The  multiplication is 
simply the concatenation of zigzags. The singleton zigzag, i.e. the 
loop, is the least element as well as the neutral element. Coding 
forward arcs by $"+"$ and backward arcs by $"-"$ we can identify 
zigzags with words (i.e., finite sequences) over the alphabet $\{+,-\}$. 
For example, the zigzag from $a$ to $b$ shown in Figure 1 receives the 
code $+++--+--$ and the reverse zigzag from $b$ to $a$ is coded by
$++-++---$. 

Hence the ordered monoid of zigzags is nothing else but the 
$2$-generated free monoid $\{+,-\}^{\ast}.$ The ordering of 
$\{+,-\}^{\ast}$ is  the {\it subword} or {\it divisibility\ 
ordering}. Now, the {\it distance} $d(a,b)$ from vertex $a$ to vertex 
$b$ in any oriented graph $G$ is the "upper" subset of 
$\{+,-\}^{\ast}$ consisting of all words coding zigzags which map 
homomorphically to a subzigzag of $G$ from $a$ to $b$. Every upper 
set $Z$ of $\{+,-\}^{\ast}$ (such that every word above some element 
of 
$Z$ belongs to $Z$) may occur as a distance in some oriented graph 
except the set of all nonempty words. This exception simply reflects 
the hypothesis of antisymmetry, for, if $+$ as well as $-$ belong to 
$d(a,b)$, then  we would get a double arc between $a$ and $b$ unless 
$a=b$. The set of all words (being the upper set generated by the 
empty word) then  constitutes the "zero" 
distance.

 An {\it isometry} (or isometric embedding) of $G$ into 
another oriented graph $H$ is a mapping $f$ from $G$ to $H$ preserving 
distances, i.e. $d(f(a), f(b))= d(a,b)$, and is necessarily injective and preserves arcs. The 
isometric embedding in products of zigzags is governed by the Galois 
connection induced by the ordering between zigzags (or words), viz., 
the MacNeille completion of $\{+,-\}^{\ast}$. First observe that the 
"lower cone" $d(a,b)^{\nabla}$ formed by the words below all words in 
a distance $d(a,b)$ of $G$ has an obvious interpretation: it 
consists of the words coding zigzags $P$ for which there are 
arc-preserving mappings from $G$ onto $P$ sending $a$ and $b$ to the 
initial and terminal vertices of $P$, respectively. So, if $a: = 
(a_{i})$ and $b: = (b_{i})$ are vertices in a product of zigzags, then 
the words coding the zigzags from the $a_{i}'s$ to the $b_{i}'s$ in 
the factors are exactly the members of the lower cone of the distance 
from $a$ to $b$. This merely rephrases the universal property of 
products in our category. Hence the distance from $a$ to $b$ in a 
product is a "closed"  upper set, and therefore every oriented graph 
$G$ which isometrically embeds in such a product has MacNeille closed 
distances. 

That the converse is also true, is affirmed by a (more 
general) result of Jawhari, Misane and  Pouzet \cite {JaMiPo}, Proposition IV-4.1. This 
characterization is, however, not yet completely satisfactory because 
of its partially extrinsic nature: how would we check that a 
distance is closed other than by computing the lower cone, which 
consists of the words coding potential zigzag factors? Fortunately, 
there is an intrinsic way to verify closedness: assume $y+z$ and 
$y-z$ are any two words in a distance $d(a,b)$ with common circumfix 
$y \ldots z$ but different infix letters; then the common subword 
$yz$ must belong to $d(a,b)$ whenever this distance is closed. We 
refer to this checking procedure as the "cancellation rule". To give 
an example, assume that a closed upper set $Z$ of words contains $+++$ 
and $---$. Then, as $Z$ is an upper set, we have both $+-+-+$ 
and $--+-+$ in $Z$, whence $-+-+$ belongs to $Z$ by the cancellation 
rule. Further, as $++-+$ is in $Z$, the rule applied to the latter 
two words returns $+-+$, and since also $+++$ is in $Z$, so must be 
the common subword $++$. Interchanging the roles of $+$ and $-$ we 
infer that $Z$ contains $--$. Repeating essentially the same argument 
for the words $++$ and $--$ proves that $+$ and $-$ are words in $Z$, 
whence $Z$ contains the empty word and thus is all of $\{+,-\}^{\ast}.$

The category of oriented graphs  is not the only one where 
embeddability in products of certain paths or zigzags can be 
characterized by closedness of distances. Suitable coding schemes are 
then necessary in order to capture the kind of adjacency relation for 
pairs of vertices. Consider, for instance, undirected multigraphs 
(with loops of unbounded multiplicity) and mappings that do not 
decrease the multiplicity of edges. Adjacency is coded by the 
multiplicity of the corresponding edges, and thus the alphabet 
$\Lambda$ is linearly ordered here (as the negative integers). If we 
want to deal with arbitrary directed graphs (with loops), then the 
appropriate alphabet consists of three letters $+,-, \sharp$ coding 
for backward, forward, and two-way arcs, respectively. This alphabet 
is necessarily ordered so that the letter $\sharp$ is below $+$ 
and $-$, which exactly expresses the fact that a two-way arc entails a 
forward and a backward arc; see Figure 2.

\begin{figure}[h]
    \centering
        \includegraphics[width=0.6\textwidth]{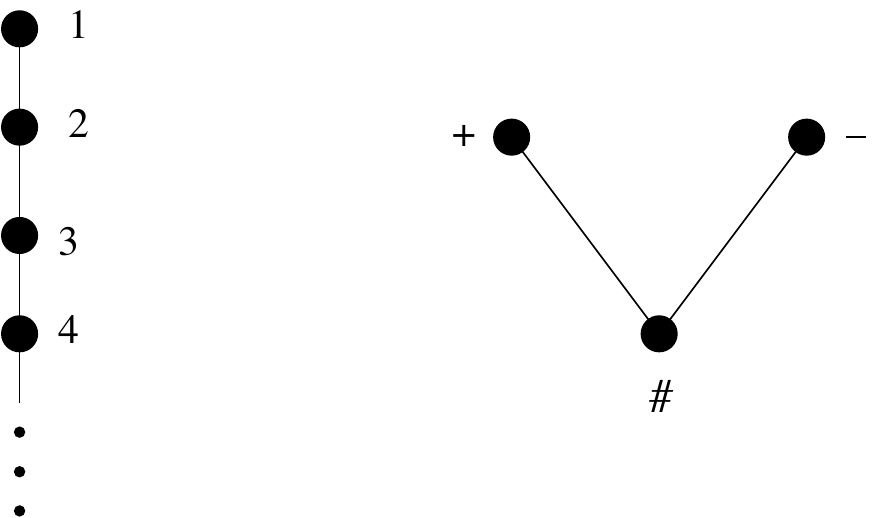}
           \caption{Ordered alphabets for undirected multigraphs and directed 
graphs, respectively}
              \label{O(eta)}
\end{figure}

This motivates the use of ordered alphabets. The ordering of letters 
extends (freely) to an ordering of words, viz., the Higman ordering 
of $\Lambda^{\ast}$ (refining the divisibility ordering). Then, with 
the right notions of product and path/zigzag, we obtain analogous 
embedding theorems for multigraphs and directed graphs.

The next section provides the necessary details on the free ordered 
monoid $\Lambda^{\ast}$ (over an ordered alphabet $\Lambda$) and
its MacNeille  completion. Under some additional assumption on 
$\Lambda$ the closed lower sets of $\Lambda^{\ast}$ other than $\Lambda^{\ast}$ are exactly the 
lower sets generated by finitely many words,  see Section 3.  A syntactical description of the 
upper closure (in the MacNeille completion of $\Lambda^{\ast}$) is 
established for particular classes of ordered sets $\Lambda$ in Section 4 (see Theorem \ref{theo:main}). This applies to the ordered alphabets coding arcs in oriented graphs or multigraphs with bounded multiplicity of edges, respectively, but not to the ordered alphabets displayed in Figure 2.

\section{The free ordered monoid and its completion}
An \emph{alphabet} is a  not necessarily finite, ordered set  $\Lambda$. Its  elements are \emph{letters} and denoted by small Greek letters  $\alpha, 
\beta, \gamma, \delta, \lambda$ etc.  A finite sequence 
$(\alpha_{1}, \ldots, \alpha_{m})$ of letters is a \emph{word} of length $m$ and is written as $\alpha_{1} \alpha_{2}\ldots \alpha_{m}$. The word of length $0$ is the \emph{empty} word, denoted by $\Box$. The words of length $1$ are identified with the corresponding letters. The concatenation of two words $x:= \alpha_{1} \alpha_{2}\ldots \alpha_{m}$ and $y:=\beta_{1} \beta_{2}\ldots \beta_{n}$ is the word $xy$ given by 
$$xy:= \alpha_{1} \alpha_{2}\ldots \alpha_{m}\beta_{1} \beta_{2}\ldots \beta_{n}.$$

\noindent For each $i$ with 
$1 \leq i \leq m$, the word $\alpha_{1} \ldots \alpha_{i}$ is 
a {\it prefix} of $x=\alpha_{1} \ldots \alpha_{m}$ while the word 
$\alpha_{i} \ldots \alpha_{m}$ is a \emph{suffix} of $x$.

 The set $\Lambda^{\ast}$ of all words is a \emph{monoid} with respect to concatenation, where the empty word is the neutral element.  The 
order relation of $\Lambda$, denoted by $\leq$, extends to $\Lambda^*$
 in the following way: 
$$x = \alpha_{1} \alpha_{2} \ldots \alpha_{m} \leq y = \beta_{1} 
\beta_{2} \ldots \beta_{n}$$
if and only if 
$$\alpha_{j} \leq \beta_{i_{j}}\ {\rm for\ all}\ j = 1, \ldots m\; \text{with some}\; 1\leq i_1<\dots i_m\leq n.$$
That is, $x$ is below $y$ in $\Lambda^*$ exactly when there exists a subword 
$\beta_{i_{1}} \beta_{i_{2}} \ldots \beta_{i_{m}}$ of $y$ which is letter-wise above $x$ (in the ordering of $\Lambda^*$). Then $\Lambda^*$ becomes an ordered monoid (i.e., $x\leq y$ implies $xz\leq yz$ and $zx\leq zy$) in which the empty word is the least element. The ordered monoid $\Lambda^*$ is freely generated by the ordered set $\Lambda$, see \cite{cohn}, that is,  $\Lambda^*$ is the free object in the category of ordered monoids (whose neutral elements are also the least elements) and order-preserving homomorphisms.

The ordered monoid $\Lambda^*$ can be extended to a complete lattice ordered monoid by applying the MacNeille completion. The necessary notation (cf. Skornjakow \cite{Sko}) is introduced next. 
 Let $X$ be a subset of 
$\Lambda^{\ast}$; then $$\uparrow X:= \{ y \in \Lambda^{\ast}: x\leq y\;  \text{for some}\;  x\in X\}$$  is the {\it upper set} generated by $X$ and 
$$\downarrow X := \{x \in \Lambda^{\ast}:  x\leq y\;  \text{for some}\;  y\in X\}$$   is the {\it lower set} generated by $X$. Upper sets and lower sets 
are \emph{finitely generated} if they are of the form $\uparrow X$, resp. $\downarrow X$ for some finite set $X$. For a singleton $X = 
\{x\}$, we omit the set brackets and call $\uparrow x$ and $\downarrow 
x$ a  \emph{principal upper set} and a \emph{principal lower set}, respectively. Then
$$X^{\Delta}:= \bigcap_{x \in X} \uparrow x$$
 and 
$$X^{\nabla}:= \bigcap_{x \in X} \downarrow x$$
are the {\it upper cone} and  the {\it lower cone} respectively,  generated by $X$. 

The pair $(\Delta, \nabla)$ of mappings on $\powerset(\Lambda^*)$,  the power set lattice of $\Lambda^*$, 
constitutes a Galois connection, yieldings the \emph{MacNeille completion} of $\Lambda^*$. This completion is realized  as the complete 
lattice $$\{W \subseteq \Lambda^{\ast} : W = W^{\Delta \nabla}\}$$ 
ordered by inclusion or its isomorphic copy  $$\{Y \subseteq \Lambda^{\ast}: Y = 
Y^{\nabla \Delta}\}$$ ordered by reverse inclusion. The members of those two sets are said to be (MacNeille) \emph{closed}.   The set $\Lambda^{\ast}$ 
embeds into the former set via $x \mapsto \downarrow x$ and into the 
latter via $x \mapsto \uparrow x$ ($x\in \Lambda^*$).

To give an example, consider the alphabet $\Lambda := \{+,-\}$ 
where $+$ and $-$ are incomparable letters. Then:

$$\{-+-+-, +-+-+, +--+-\}^{\nabla} =\; \downarrow \{--+, +-+-\}$$ and 
$$\{- \;-+, +-+-\}^\Delta =\; \uparrow \{-+-+-, +-+-+, +-\;-+-\},$$

\noindent showing that the latter upper set $\uparrow\{-+-+-, +-+-+, +--+-\}$ is closed. In 
contrast,  $\uparrow\{+,-\}$ is not closed since $\{+,-\}^{\nabla 
\Delta} = \Lambda^{\ast}.$

The completion of $\Lambda^*$ inherits its monoid structure from the power set, where the muliplication is given by 
$$XY: = \{xy : x \in X, y \in Y\}$$  for any subsets $X$ and $Y$ of $\Lambda^*$. The cone operators preserve this multiplication as the following lemma confirms.

\begin{lem}\label {lem1}
For any subsets  $X,Y$ of  $\Lambda^*$,   
$$(XY)^{\nabla} = X^{\nabla}Y^{\nabla}, \ {\rm and}\ (XY)^{\Delta} = 
X^{\Delta} Y^{\Delta}\;  \text {if}\;  X, Y\not =\emptyset, $$
whence
$$(XY)^{\nabla \Delta} = X^{\nabla \Delta}Y^{\nabla 
\Delta}\ {\rm and}\ (XY)^{\Delta \nabla} = X^{\Delta \nabla}
Y^{\Delta \nabla}$$
\end{lem}

\begin{proof}
First, observe that $\emptyset^{\nabla}= \emptyset ^{\Delta}= \Lambda^*$ and $\Lambda^{* \Delta}=\emptyset$, while $\Lambda^{* \nabla}$ consists of the empty word. Further, $\emptyset Z=Z \emptyset=\emptyset$ for every subset $Z$ of $\Lambda^*$. The inclusions $X^{\nabla} Y^{\nabla} \subseteq (XY)^{\nabla}$ and
$X^{\Delta} Y^{\Delta} \subseteq (XY)^{\Delta}$ are then immediate.

Suppose that there exists a word $w$ in $(XY)^{\nabla}$ that does not 
belong to $X^{\nabla} Y^{\nabla}$.  Then let $u$ be the longest prefix 
of $w$ from $X^{\nabla}$, and let $v$ be the longest suffix 
of $w$ from   $Y^{\nabla}$ so that  $w$  is of the form
$$w =u {\alpha_{1}} \ldots \alpha_{k} v$$
for some letters $\alpha_{1}, \ldots, \alpha_{k}$, where $k \geq 1$. 
By the choice  of $u$ and $v$, there are words $x \in X$ and $y \in Y$ 
such that
$$u \alpha_{1} \not \leq x\ {\rm and}\ \alpha_{k} v \not \leq y.$$
This, however, is in conflict with
$$w = u \alpha_{1} \ldots \alpha_{k} v \leq xy.$$
Therefore $(XY)^{\nabla}$ equals $X^{\nabla} Y^{\nabla}.$
Finally, suppose  $z$ is a word in $(XY)^{\Delta}$ which does not 
belong to $X^{\Delta} Y^{\Delta}$, where $X$ and $Y$ are nonempty. Then the shortest  prefix 
of $z$ from $X^{\Delta}$ and the shortest suffix of $z$ from 
$Y^{\Delta}$ intersect in a nonempty subword
$$w:= \alpha_{1} \ldots \alpha_{k}$$ 
so that $z$ can be written as 
$$z = uwv\ {\rm with}\ uw \in X^{\Delta}\ {\rm and}\ wv \in 
Y^{\Delta}.$$
By the choice of the words $u$ and $v$, we can find words $x \in X$ and $y \in 
Y$ with  
$$x \not \leq u \alpha_{1} \ldots \alpha_{k-1}\ {\rm and}\ y \not \leq 
v.$$
This contradicts the hypothesis that 
$$xy \leq z = u{\alpha_{1}} \ldots \alpha_{k-1} \alpha_{k} v.$$
We conclude that $(XY)^{\Delta} = X^{\Delta} Y^{\Delta}$, completing the proof.  
\end{proof}\\

The completion of $\Lambda^*$, realized by the upper closed sets, is a complete lattice in which suprema are  set-theoretic intersections, whereas infima are the closures of set-theoretic unions. The   \emph{closed union} of a family $Z_i \; (i\in I)$ of upper sets in $\Lambda^*$ is given by:
$$ \bigsqcup_{i\in I}Z_i= (\bigcup_{i\in I} Z_i)^{\nabla\Delta}.$$
The following result entails that the completion of $\Lambda^*$ is a  complete latttice ordered monoid (in the sense of Birkhoff \cite{birkhoff}). 
\begin{proposition}\label{cor}
For any ordered alphabet $\Lambda$, the collection  of all closed upper sets of 
words over $\Lambda$ is a monoid and   complete lattice such that 
the multiplication distributes over intersection  and closed unions, that is
$$Y(\bigcap_{i \in I} Z_{i}) = \bigcap_{i \in I} YZ_{i}\ {\rm and}\ 
(\bigcap_{i \in I} Z_{i}) Y = \bigcap_{i \in I} Z_{i} Y,$$
$$Y(\bigsqcup_{i \in I} Z_{i}) = \bigsqcup_{i \in I} YZ_{i}\ {\rm and}\ 
(\bigsqcup_{i \in I} Z_{i}) Y = \bigsqcup_{i \in I} Z_{i}Y$$
for any index set $I$ and all closed upper sets $Y, Z_{i}\;  (i \in I).$\\
\end{proposition}

\begin{proof}
\noindent Since $Y$ and all $Z_i$ are closed and $(\Delta, \nabla)$ is a Galois connection, we have 
$Y= Y^{\nabla\Delta}$ and $\bigsqcup_{i\in I}Z_i= (\bigcup_{i\in I} Z_i)^{\nabla\Delta}= (\bigcap_{i\in I} Z_i^{\nabla})^{\Delta}$. Analogous formulae hold  for $Y^{\nabla}$ and $Z_{i}^{\nabla}\; (i\in I)$. Hence by Lemma \ref {lem1}
$$Y(\bigsqcup_{i \in I} Z_{i})= Y^{\nabla\Delta}(\bigcup_{i \in I} Z_{i})^{\nabla \Delta} = [Y 
(\bigcup_{i \in I} Z_i)]^{\nabla \Delta} =$$ $$= ( \bigcup_{i \in I} Y Z_{i})^{\nabla\Delta}=  \bigsqcup_{i \in I} (Y Z_{i})^{\nabla\Delta}= \bigsqcup_{i \in I} Y^{\nabla\Delta}Z_i^{\nabla\Delta}=  \bigsqcup_{i \in I} YZ_i.$$
\noindent  Further
$$Y (\bigcap_{i \in I} Z_{i})= 
Y^{\nabla\Delta}(\bigcap_{i \in I} Z_i)^{\nabla\Delta} = (Y^{\nabla}\bigcup_{i \in I}  Z_i^{\nabla})^{\Delta}= (\bigcup_{i \in I} Y_{i}^{\nabla} Z_i^{\nabla})^ {\Delta} = \bigcap_{i \in I}Y^{\nabla\Delta} 
Z_i^{\nabla \Delta}= \bigcap_{i\in I}YZ_i.$$  
 
\noindent This settles left distributivity; the proof of 
right distributivity is analogous. \end{proof}\\

Note that the  collection  of all closed upper sets of 
words over $\Lambda$ is in fact a free monoid, see \cite{kabil-pouzet-rosenberg}.

\section{Closed lower sets}
In this section, we describe the closed lower sets.
The 
characterization of the closed upper sets is more involved and shall 
occupy us for the rest of the next section.

Finiteness assumptions on the 
alphabet $\Lambda$ allow to argue by induction or to obtain finite generation. $\Lambda$ is said 
to be {\it well-founded}  (or to satisfy the \emph{descending chain condition}, DCC for short, \cite{birkhoff}) if $\Lambda$ does not contain any infinite 
decreasing chain $\lambda_{0} > \lambda_{1} > \ldots .$ Then 
$\Lambda$ is called {\it well-quasi-ordered} if $\Lambda$ is 
well-founded and has no infinite antichain (that is, contains no 
infinite subset of pairwise incomparable elements). We recall a fundamental result. 

\begin{theo}(Higman \cite{Hi})\label{theowqo} If $\Lambda$ is  well-quasi-ordered then $\Lambda^{\ast}$ is well-quasi-ordered too,  
whence the complete lattice of all lower sets of 
$\Lambda^{\ast}$ is well-founded, which means   that every upper set in 
$\Lambda^{\ast}$ is finitely generated. 
\end{theo}

For a proof see also Nash-Williams \cite{nashwilliams}  or Cohn \cite{cohn}. For 
more information on well-quasi-ordered sets, see the survey paper of 
Milner \cite{milner}. 
If $\Lambda$ is well-quasi-ordered, then by virtue of Higman's Theorem the MacNeille completion of $\Lambda^{\ast}$ as realized within the complete lattice of all lower sets is necessarily well-founded. However, this completion can be well-founded even when $\Lambda$ contains infinite antichains. Well-founded dual forests constitute pertinent examples, as will be seen next. 

An ordered set $\Lambda$ is called an {\it ordered tree} if every 
principal lower set of $\Lambda$ is a chain and $\Lambda$ is 
down-directed (that is, any two elements of $\Lambda$ are bounded 
below). An {\it ordered forest} is a disjoint union of ordered trees. 
The {\it dual} (alias \emph{opposite}) of an ordered set  is obtained by reversing the order relation. 
Observe that the dual of an ordered forest (a \emph{dual forest},  for short)  is just an ordered set in which any two 
 incomparable elements are {\it incompatible}, i.e., not
bounded below (Figure \ref{forest}). 

\begin{figure}[h]
    \centering
        \includegraphics[width=0.6\textwidth]{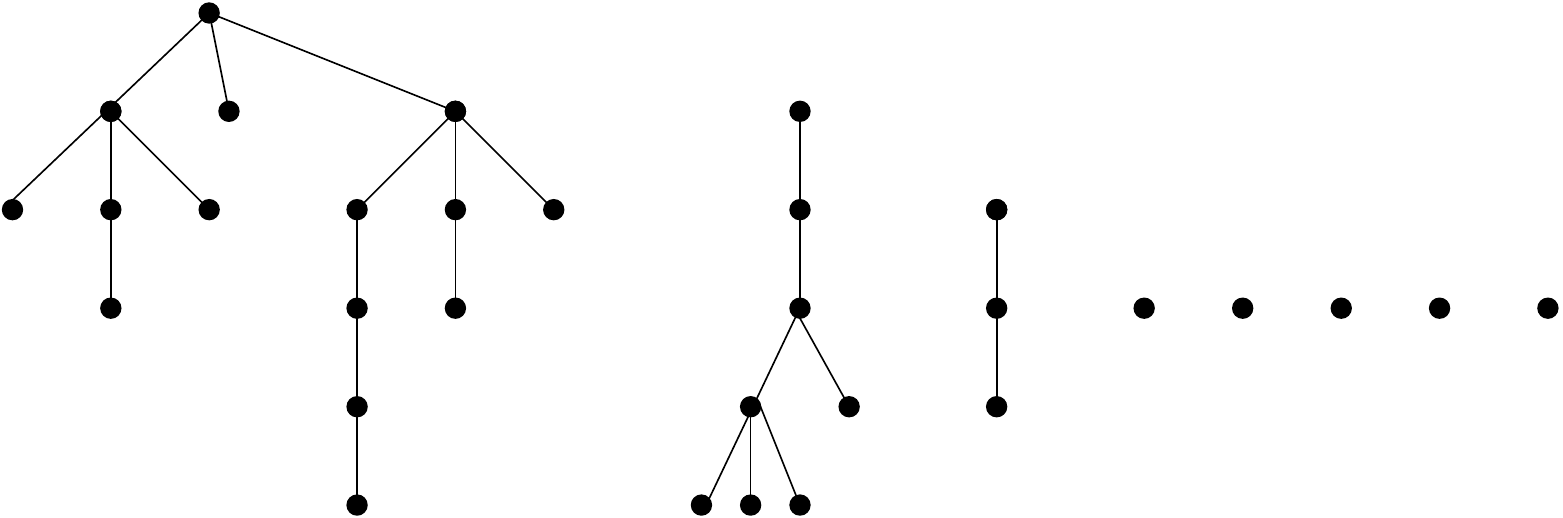}
           \caption{A (finite) dual  forest}
              \label{forest}
\end{figure}

Consider an ordered set $\Lambda$ that is 
well-quasi-ordered and the dual of an ordered forest.  
Since $\Lambda$ is well-founded, every element is above some minimal 
element. Let $K$ be the subset of $\Lambda$ consisting of all 
existing joins of minimal elements. Then, as $\Lambda$ has no 
infinite antichain, $K$ is  a dual  finite  forest. 
It is not difficult to see that $\Lambda$ is the lexicographic sum of 
a family of ordinals indexed by $K$. Note that adding a least 
element to $\Lambda$ (if necessary) results in a complete lattice.

It will turn out (see Theorem \ref{thm:lowersets} below) that the finitely 
generated lower sets of $\Lambda^{\ast}$ together with 
$\Lambda^{\ast}$ are exactly the closed lower sets in any well-founded dual forest $\Lambda$. Two lemmas are needed to establish this.

\begin{lem}\label{lem2.4} Let $\Lambda$ be an ordered set. Then all (MacNeille) closed lower 
sets of $\Lambda^{\ast}$ different from $\Lambda^{\ast}$ are  finitely generated lower sets if and only if  $\Lambda$ is well-founded and the intersection of any two 
principal lower sets of $\Lambda$ is a finitely generated lower set. Hence, in this case, the MacNeille completion of  $\Lambda^{\ast}$ is necessarily well-founded. 
\end{lem}

\begin{proof} $\Lambda$ can be regarded as a lower set of 
$\Lambda^{\ast}\setminus \{\Box\}$. Therefore, if $\Lambda^{\ast}$ is well-founded, so is $\Lambda$. The intersection of any two principal lower sets  in $\Lambda$ is closed, whence this is a finitely generated lower set under the asssumption that all closed lower sets in $\Lambda^{\ast}$ other than $\Lambda^{\ast}$ be finitely generated. Restricting this intersection to $\Lambda$ amounts to removing the empty word. This establishes necessity of the conditions on $\Lambda$.

To prove the  converse, assume that $\Lambda$ is well-founded such that any two principal lower sets of $\Lambda$  intersect in a finitely generated lower set. We extend the original alphabet $\Lambda$  to the set $\overline \Lambda:=   \Lambda \cup \{ \Box\}$, where $\Box$  becomes the least element of the extended alphabet. The set   ${\overline \Lambda }^{\ast}$ of words   over $\overline{\Lambda}$ is the union of $\overline {\Lambda}^{n}$  for all $n\geq 0$ (here the empty word is distinct from $\Box$). There is a canonical map $\varphi$ from $\overline \Lambda ^{\ast}$ onto $\Lambda^{\ast}$ which  "forgets"  the empty letter $\Box$, viz.,  $\varphi$  maps the empty tuple to ${\Box}$ and  a nonempty tuple $x$ from $\overline {\Lambda}^{n}$( $n>0$) to the concatenation of its coordinates with respect to the indexing order. For instance, both tuples $(\Box,+,\Box,-)$ and $(+,-,\Box)$ from $\overline{ \{+,-\}}^{\ast}$ are mapped to $+-$ under $\varphi$. Thus, the pre-images under $\varphi$ of a fixed word differ only in the number and positions of the empty letter $\Box$. The map  $\varphi$  obviously is a monoid homomorphism such that  for any two words $w$ and $x$  in $\Lambda^{\ast}$ we have $w < x$ exactly when for each tuple $x' $ in the pre-image of $x$ under $\varphi$  there exist a tuple $w'$  in the pre-image of $w$ such that $w' < x'$. 

First, we claim that $\Lambda^{\ast}$ is well-founded because $\Lambda$ is. Trivially, $\overline \Lambda$  is well-founded and hence so any of its finite Cartesian powers $\overline \Lambda ^n$ ($n > 0$) by virtue of the pigeonhole principle. If there was an infinite descending chain $x_0 > x_1 > x_2 >\dots$  in $\Lambda^{\ast}$  starting with some word $x_0$ of length $n$, then we could lift this chain to $\overline \Lambda ^n$ by selecting $x'_{0} = x_0$  and successively choosing tuples in  $\overline {\Lambda} ^n$  with $x'_{1 }< x'_{0},  x'_{2} < x'_{1}$, etc., contrary to the observation that  $\overline \Lambda ^n$ is well-founded.

	Second, we assert that any two principal lower sets  $\downarrow\!\! w$ and $\downarrow\!\! x$  of $\Lambda^{\ast}$  intersect in a finitely generated lower set. This is true in the particular case that $w$ and $x$ belong to $\overline \Lambda$ because of the corresponding property assumed for the alphabet $\Lambda$. If $w':= (w_1, \dots, w_n)$ and $x' := (x_1,\dots,x_n)$ belong to $\overline {\Lambda}^n (n > 0)$, then $\downarrow\!\! w'\;   \cap \downarrow\!\! x'$ is simply the Cartesian product of $\downarrow\!\! w_i \; \cap \downarrow  x_i $ for $i = 1,...n$, whence as a product of finitely generated lower sets of 
$\overline\Lambda$  it is a finitely generated lower set of  $\overline {\Lambda}^n$.   Now, if $w$ and $x$  are words of length at most $n$, then we can take corresponding tuples $w'$ and $x'$  in $\overline {\Lambda} ^n$  which are mapped to $w$ and $x$ by $\varphi$. Since  $\varphi  $ maps lower sets onto lower sets, we infer that $\downarrow w\; \cap \downarrow x$  is a finitely generated lower set of $\Lambda ^{\ast}$.

	Third, we claim that for every finite subset $Z$ of  $\overline \Lambda^{\ast}$, the lower cone $Z^{\nabla}$ is a finitely generated lower set. If $Z$ has cardinality at most $2$, this has just been established. Now, by an induction hypothesis, for any $y$ in $Z$, there is a finite antichain $X$ in  $\Lambda^{\ast}$  such that $(Z\setminus \{y\})^{\nabla} = \; \downarrow X$. 
Then $\downarrow \!X\;  \cap \downarrow\!\!y$  equals the union of all $\downarrow\!\!x\;  \cap  \downarrow\!\!y$ for $x$ from $X$ and thus is a finitely generated lower set, as required.

 Fourth, a result of Birkhoff \cite {birkhoff}, Theorem 2, p. 182,  states that the set of 
finitely generated lower sets of any well-founded ordered set $P$ is 
well-founded. Hence, the set of finitely generated initial segments of $\Lambda^{\ast}$ is well-founded. 
From this and the well foundedness of $\Lambda^{\ast}$ we derive that every closed lower set $X$ other than $\Lambda^{\ast}$ is 
finitely generated. Consider the collection of all lower cones of the 
form $Z^{\nabla}$ where $Z$ is a finite subset of the upper cone 
$X^{\Delta}$ (so that $X \subseteq Z^{\nabla}$). This collection is 
nonempty because $X \neq \Lambda^{\ast}$, and it contains some minimal 
member $Y^{\nabla}.$ Suppose we could find $w \in Y^{\nabla} \setminus  X$. 
Then $w \not \leq z$ for some $z \in X^{\Delta}$ because $X$ is 
closed. Now, by minimality of $Y^{\nabla}$ we have
$$w \in Y^{\nabla} = (Y \cup \{z\})^{\nabla} \subseteq \downarrow z,$$
giving a contradiction. This completes the proof. \end{proof}\\

Note that  every closed upper set of $\Lambda^{\ast}$ is of the form 
$Y^{\nabla \Delta}$ for some finite subset $Y$ whenever the 
MacNeille completion of $\Lambda^{\ast}$ is well-founded. Observe 
that $Y^{\nabla \Delta}$ need not be a finitely generated upper 
set. Lemma \ref{lem2.4} applies, in particular, to a well-founded conditional 
lattice $\Lambda$ (such as a well-founded dual forest), yielding the 
finiteness conditions for the MacNeille completion of $\Lambda^{\ast}$. 
Here we say that an ordered set $\Lambda$ is a {\it conditional 
lattice} if it is obtained from a bounded lattice by removing the 
bounds. In other words, $\Lambda$ is a conditional lattice if and only 
if every pair of elements bounded below has a meet and every pair of 
element bounded above has a join.

\begin{lem}\label{lem:down-closed} Let $\Lambda$ be an ordered set. Then every finitely 
generated lower set in $\Lambda^{\ast}$ is (MacNeille) closed if and only if each 
pair of letters from $\Lambda$ that is bounded below is also bounded 
above.
\end{lem}

\begin{proof} Let $\alpha, \beta, \lambda$ be letters such that $\lambda < 
\alpha, \lambda < \beta,$ but $\alpha, \beta$ do not have an upper 
bound. Consider the lower set $W := \downarrow\{\alpha, \beta\}$ in 
$\Lambda^{\ast}$. Since $\{\alpha, \beta\}$ is not bounded above, 
every word above $\alpha$ and $\beta$ is above $\alpha \beta$ or 
$\beta \alpha$. Hence $W^{\Delta} = \uparrow\{\alpha \beta, \beta 
\alpha\}.$ Then the word $\lambda \lambda$ belongs to $W^{\Delta 
\nabla}$ but not to $W$, showing that $W$ is not closed.

Conversely assume that $\Lambda$ satisfies the condition of the lemma. 
Let $w,x,y$ be words in $\Lambda^{\ast}$ such that $w$ does not belong to $\downarrow \{x,y\}$, that is,
$$w \not \leq x\ {\rm and}\ w \not \leq y.$$
We claim that $w\not \in \{x, y\}^{\Delta \nabla}$, that is,  there exists a word $z$ such that
$$x \leq z, y \leq z, {\rm and}\ w \not \leq z.$$
Assume that $w = \alpha_{1} \ldots \alpha_{n}$ with $\alpha_{i} \in \Lambda$. 
Let $x_{n}$ be the (possibly empty) largest suffix of $x$ consisting only of 
letters not above $\alpha_{n}$. If $\alpha_{n} \not \leq x,$ then 
$x_{n} = x.$ Otherwise, there exists some $\beta_{n} \in \Lambda$
such that $\alpha_{n} \leq 
\beta_{n}$ and $x$ is of the form $x = u \beta_{n} x_{n}$ for some 
(possibly empty) word $u$. Then $n\geq 2$ and $\alpha_{1} \ldots \alpha_{n-1} \not 
\leq u$.  We continue as before, so that we 
eventually obtain a representation of $x$ as
$$x = x_{i} \beta_{i+1} x_{i+1} \ldots x_{n-1} \beta_{n} x_{n},$$
where $1 \leq i \leq n$, and
$$\alpha_{k} \not \leq x_{k}\ {\rm for}\ i \leq k \leq n,$$
$$\alpha_{k} \leq \beta_{k}\ {\rm for}\ i < k \leq n.$$
Then $\alpha_{i+1} \ldots \alpha_{n}$ is the largest suffix of $w$ 
that is below some subword $\beta_{i+1} \ldots \beta_{n}$ of $x$, 
where in addition $\beta_{i+1} \ldots \beta_{n}$ is the right-most 
subword of $x$ with this property. Similarly, we have a representation
$$y = y_{j} \gamma_{j+1} y_{j+1} \ldots y_{n-1} \gamma_{n} y_{n},$$
where $1 \leq j \leq n$, and
$$\alpha_{k} \not \leq y_{k}\ {\rm for}\ j \leq k \leq n,$$
$$\alpha_{k} \leq \gamma_{k}\ {\rm for}\ j < k \leq n.$$
We may assume that $i \leq j$. Now, by the condition on $\Lambda$, we 
can find a letter $\lambda_{k}$ such that
$$\beta_{k} \leq \lambda_{k}\ {\rm and}\ \gamma_{k} \leq 
\lambda_{k}\ {\rm whenever}\ j \leq k \leq n.$$
Put
$$z = x_{i} \beta_{i+1} \ldots x_{j-1} \beta_{j} x_{j} y_{j} 
\lambda_{j+1} \ldots x_{n-1} y_{n-1} \lambda_{n} x_{n} y_{n}.$$
Then $x \leq z$ and $y \leq z$, but $w \not \leq z$ by the choice of the 
$x_{k}$ and $y_{k}.$ This proves the claim. Now, by a trivial 
induction we get that for any words $w, x_{1}, \ldots, x_{m}$ 
with $w \not \leq x_{k}$ for all $k$ there exists a word $z$ such 
that $w \not \leq z$ and $x_{k} \leq z$ for all $k$. So, if $X = 
\downarrow\{x_{1}, \ldots, x_{m}\}$ is some finitely generated set 
and $w \in \Lambda^{\ast}\setminus X$, then there exists $z \in \{x_{1}, 
\ldots x_{m}\}^{\Delta}$ such that $w \not \in \downarrow z$, that 
is, $w \not \in X^{\Delta \nabla}.$ This proves that $X^{\Delta 
\nabla} \subseteq X$, whence $X$ is closed. \end{proof}\\

The preceding lemma covers the result of Jullien \cite{jullien} for unordered 
finite alphabets (i.e., in the case that $\Lambda$ is a finite 
antichain); see also Kabil and Pouzet \cite{KaPo1},  Proposition 2.2.

Recall that a pair of  elements $\alpha, \beta\in \Lambda$ is \emph{compatible} if these elements have a common lower bound. 
\begin{theo}\label{thm:lowersets} 
Let $\Lambda$ be an ordered set.  Then the  (MacNeille) closed lower sets of $\Lambda^{\ast}$ form a well-founded lattice which exactly comprises $\Lambda^{\ast}$  and all finitely generated lower sets if and only if $\Lambda$ is well-founded and every compatible pair of (incomparable) elements $\alpha, \beta\in \Lambda$ is   bounded above and the common lower bounds of $\alpha,  \beta$ form a finitely generated lower set. In particular, the ordered set $\Lambda$ obtained from some disjoint union of well-founded lattices by removing the antichain of minimal elements is of this kind. 
\end{theo}

\begin{proof} From Lemma \ref{lem2.4} we infer that every closed lower set $W \neq 
\Lambda^{\ast}$ in $\Lambda^{\ast}$ is finitely generated. 
Conversely, a finitely generated lower set of $\Lambda^{\ast}$ is 
closed by Lemma \ref{lem:down-closed} since $\Lambda$ satisfies the hypothesis of this 
Lemma. \end{proof}

\section{Closed upper sets}

Given a set $Y$ of words over an ordered set $\Lambda$, we wish to 
build up its closure $Y^{\nabla \Delta}$ by successively applying 
a few (partial) binary operations and taking upper sets (which, of 
course,  is governed by a family of unary operations indexed by 
$\Lambda^{\ast}$). Certainly, one cannot circumvent some finiteness 
condition on $Y$ as the MacNeille completion is inherently 
infinitary. Since we reserve the name ''closed upper set'' for 
members of this completion, we say that $Z \subseteq \Lambda^{\ast}$ 
is {\it stable} with respect to a partial operation $f 
$ defined on 
$D(f) \subseteq \Lambda^{\ast}  \times \Lambda^{\ast}$ if
$$(x,x') \in D(f) \bigcap (Z \times Z)\ {\rm implies}\ f(x,x') \in Z.$$

\begin{lem}\label{lem:allrules}
Let $\Lambda$ be an ordered set. Then every closed upper set $Z$ in 
$\Lambda^{\ast}$ is stable with respect to the four partial binary 
operations ''cancellation'', ''reduction'', ''permutation'', and 
''meet'':\\

\noindent{\it (cancellation rule)} if $y \alpha z \in Z$ and $y \beta 
z \in Z$ where $\alpha, \beta$ are incompatible letters (that is, not bounded below) and $y,z \in 
\Lambda^{\ast}$, then $y z \in Z;$\\

\noindent{\it (reduction rule)} if $y \alpha \alpha z \in Z$ and $y 
\gamma z \in Z$ for $\alpha < \gamma$ in $\Lambda$ and $y,z \in 
\Lambda^{\it }$, then $y \alpha z \in Z;$\\

\noindent{\it (permutation rule)} if $y \alpha \beta z \in Z$ and $y 
\gamma z \in Z$ where $\alpha, \beta, \gamma \in \Lambda$ and $y, z 
\in \Lambda^{\ast}$ such that $\alpha, \beta$ are incomparable and 
below $\gamma$, then $y \beta \alpha z \in Z;$\\

\noindent{ \it (meet rule)} if $y \alpha z \in Z$ and $y \beta z \in 
Z$ such that $\alpha, \beta \in \Lambda$ are incomparable letters with 
meet $\alpha \wedge \beta$ in $\Lambda$ and $y, z \in \Lambda^{\ast}$, 
then $y(\alpha \wedge \beta) z \in Z.$
\end{lem}

\begin{proof} Let $u,v,y,z \in \Lambda^{\ast}$ such that $yuz, yvz \in Z.$ Then, 
according to Lemma 2.1, we obtain
$$\{yuz, yvz\}^{\nabla} = (y\{u,v\} z)^{\nabla} = (\downarrow y) 
\{u,v\}^{\nabla}(\downarrow z)$$
and hence
$$\{yuz, yvz\}^{\nabla \Delta} = (\uparrow y) \{u,v\}^{\nabla 
\Delta} (\uparrow z).$$
Since $Z$ is a closed upper set, the preceding upper cone in included 
in $Z$, that is,
$$ywz \in Z\; \text{ for all}\; w \in \{u,v\}^{\nabla \Delta}.$$
This  applies to each of the four asserted rules. In each case the 
closure of $\{u,v\}$ is readily determined: if $\alpha$ and $\beta$ 
are incompatible, then $\{\alpha, \beta\}^{\nabla \Delta} = 
\Lambda^{\ast}$; and if $\alpha \wedge \beta$ exists then $\{\alpha, 
\beta\}^{\nabla \Delta} = \uparrow (\alpha \wedge \beta).$ For 
$\alpha < \beta$ we get $\{\alpha \alpha, \beta\}^{\nabla 
\Delta} = \uparrow \alpha.$ Finally, if $\alpha$ and $\beta$ are 
incomparable such that $\alpha, \beta < \gamma$, then
$$\{\alpha \beta, \gamma\}^{\nabla \Delta} = (\downarrow \alpha\; \cup \downarrow \beta)^{ \Delta}=\{\alpha, 
\beta\}^{\Delta} =(\uparrow \alpha \; \cap \uparrow \beta \cap \Lambda)\;  \cup \uparrow\{\alpha \beta, \beta \alpha\}, $$
which equals $\uparrow \{\alpha\beta, \beta \alpha, \alpha\vee \beta\}$ whenever the join $\alpha\vee \beta$ exists. This completes the proof of the lemma. \end{proof}\\

The final argument in the preceding proof actually yields 
an extension of the permutation rule that also entails the reduction rule, viz.\\

\noindent{\it (permuto-reduction  rule)} if $y \alpha \beta z \in Z$ and $y 
\gamma z \in Z$ where $\alpha, \beta, \gamma \in \Lambda$ and $y, z 
\in \Lambda^{\ast}$ such that $\alpha, \beta$ are incomparable and 
below $\gamma$, then $y \beta \alpha z \in Z$ and $y\delta z\in Z$ for all $\delta \in \Lambda$ with $\alpha, \beta < \delta.$\\

One can also derive the second assertion in this rule from the reduction rule: if $\alpha, \beta < \delta<\gamma$ such that $y\alpha\beta z\in Z$ and  $y 
\gamma z \in Z$ then $y\delta\delta z\in Z$ and hence $y\delta z \in Z$ by the reduction rule.

The cancellation rule and the meet-rule can be regarded as a single rule with respect to the meet in $\Lambda^{\ast}$:\\

\noindent{\it (extended meet rule)} if $y \alpha z \in Z$ and $y \beta z \in 
Z$ such that $\alpha, \beta \in \Lambda$ are incomparable letters such that their
meet $w$ in $\Lambda^{\ast}$ exists,  then $y w z \in Z.$\\

In fact, this meet exists exactly when $\alpha$ and $\beta$ either are incompatible (so that $w= \Box$) or have a meet $w= \alpha\wedge \beta$ in $\Lambda$. Hence,  any two incomparable letters  have a meet in $\Lambda^{\ast}$ if and only if  $\Lambda$ is  a {\it conditional  
meet-semilattice}, that is,  every pair of compatible elements (i.e., bounded below) has a meet.

\begin{lem}\label{lem:compound}
Let $\Lambda$ be a conditional meet-semilattice. An upper set $Z$ 
in $\Lambda^{\ast}$ is stable with respect to cancellation, reduction, 
permutation, and meet precisely when $Z$ obeys the following 
''compound'' rule : if $y \alpha_{1} \ldots \alpha_{n} z \in Z (n \geq 
1)$ and $y \beta z \in Z$ such that $y,z \in \Lambda^{\ast}$ and 
$\alpha_{i}, \beta\in  \Lambda$ with $\beta \not \leq \alpha_{i}$ for all 
$i$, then $ytz \in Z$ where $t$ is a word (possibly empty) formed by the maximal 
elements of $\{\alpha_{i} \wedge \beta\; :  i = 1, \ldots, n\;  \text{such that }\;  \alpha_{i} \wedge \beta\;\text{exists} \}$ in any order.
\end{lem}

\begin{proof} Evidently the rules described in Lemma \ref {lem:allrules} are particular 
instances of the compound rule. To prove the converse, assume first 
that there is some letter $\alpha_{i}$ incompatible with $\beta$. 
Then, as $Z$ is an upper set containing $y \beta z$, the word $y 
\alpha_{1} \ldots \alpha_{i-1} \beta \alpha_{i+1} \ldots \alpha_{n} z$ 
belongs to $Z$, whence so does $y \alpha_{1} \ldots \alpha_{i-1} 
\alpha_{i+1} \ldots \alpha_{n} z$ by virtue of the cancellation rule. 
Continuing this way we can eliminate all letters $\alpha_{i}$ from  
the subword $\alpha_{1} \ldots \alpha_{n}$ in $y \alpha_{1} \ldots 
\alpha_{n} z$ that are incompatible with $\beta$, thus resulting in 
$y \lambda_{1} \ldots \lambda_{k} z \in Z$ where $\lambda_{1} \ldots 
\lambda_{k}$ is a subword of $\alpha_{1} \ldots \alpha_{n}.$ Since $y 
\lambda_{1} \ldots \lambda_{k-1} \beta z \in Z$, the meet rule gives 
$y \lambda_{1} \ldots \lambda_{k-1} (\lambda_{k} \wedge \beta) z \in 
Z.$ Iterating this argument yields $y \mu_{1} \ldots \mu_{k} z \in Z$ 
with $\mu_{i} = \lambda_{i} \wedge \beta$ for all $i$. In a similar 
way we successively apply the reduction and permutation rules: as 
every permutation of a word is the composition of transpositions 
interchanging two consecutive letters, it suffices to manipulate the 
letters $\mu_{i}, \mu_{i+1}$ for $i=1, \ldots, k-1.$ If $\mu_{i} \leq 
\mu_{i+1}$, then both $y \mu_{1} \ldots \mu_{i-1} \mu_{i+1} \mu_{i+1} 
\ldots \mu_{k} z$ and $y \mu_{1} \ldots \mu_{i-1} \beta \mu_{i+2} 
\ldots \mu_{k} z$ belong to $Z$, whence $y \mu_{1} \ldots \mu_{i-1} 
\mu_{i+1} \ldots \mu_{k} z$ by the reduction rule. If $\mu_{i} \not 
\leq \mu_{i+1},$ then the permutation rule guarantees $y \mu_{1} 
\ldots \mu_{i-1} \mu_{i+1} \mu_{i} \mu_{i+2} \ldots \mu_{k} z \in Z.$ 
This finally, shows that $ytz$ is in $Z.$ \end{proof}\\

The next lemma we need is the analogue of Lemma  \ref{lem1} for stable sets.

\begin{lem}\label{lem:concatenation}
If $U$ and $V$ are two upper sets that are stable with respect to 
cancellation, reduction, permutation, and meet, then the 
concatenation $UV$ is stable as well.
\end{lem}

\begin{proof} Let $s,y,z \in \Lambda^{\ast}$ and $\lambda \in \Lambda$ such 
that $ysz$ and $y \lambda z$ are words for which the compound rule, 
say, would return the word $ytz$. Assume $ysz, y \lambda z \in UV.$ We 
wish to show that $ytz \in UV.$ Since $U,V$ are upper sets, we infer 
from $y \lambda z \in UV$ either $y \in U$ or $z \in V$; say, the 
latter holds. If $y \in U$, then $yt \in U$ and hence $ytz \in UV.$ So 
assume that $y$ does not belong to $U.$  Let $v$ be the  shortest suffix of $z$ belonging to $V$ such that $z$ is of the form $xv$ with $y\lambda x\in U$. Then, by the minimal choice of  $v$, the word 
$ysx$ belongs to $U$ as well (because $U$ is an upper set) and hence $ytx\in U$ as $U$ is stable.  We conclude 
that $ytz = ytxv \in UV$, as required. \end{proof}\\

We are now in position to prove the main theorem. For every subset $Y$ 
of $\Lambda^{\ast}$ let $[Y]$ denote the smallest upper set of words 
which contains $Y$ and is stable with respect to cancellation, 
reduction, permutation, and meet (as described in Lemma \ref{lem:allrules}). Then by 
this lemma  we have $[Y] \subseteq Y^{\nabla \Delta}.$

\begin{theo}\label{theo:main}
Let $\Lambda$ be an ordered set in which any two elements bounded 
below have a meet and an upper bound. Then for every finite nonempty 
subset $Y$ of $\Lambda^{\ast}$ the smallest closed upper set and the 
smallest stable upper set containing $Y$ coincide: $Y^{\nabla 
\Delta} = [Y].$ If, in addition, $\Lambda$ is well-founded, then 
the closed upper sets of $\Lambda^{\ast}$ are exactly the stable upper 
sets.
\end{theo}

\begin{proof}We will show that $[Y]$ is closed for all finite nonempty sets $Y 
\subseteq \Lambda^{\ast}$ by induction. To this end, define the total 
lenght $\parallel\!\! Y \!\!\parallel$ of $Y$ as the sum of the lengths of the words 
in $Y$. If $\parallel\!\! Y \!\!\parallel = 0$, that is, $Y$ consists only of 
the empty word, then we get $[Y] = \Lambda^{\ast}$. So let $n = 
\parallel \!\!Y\!\! \parallel \; \geq 1$, and assume that $[X]$ is closed for all 
nonempty sets $X \subseteq \Lambda^{\ast}$ with $\parallel \!\!X\!\! 
\parallel \; < n.$ If $Y$ is not an antichain, then $\uparrow Y =\;  
\uparrow X$ for some proper subset $X$ of $Y$ (giving $\parallel\!\! X 
\!\!\parallel < \parallel\!\! Y \!\!\parallel$), whence $[Y] = [X]$ is closed by 
the induction hypothesis. Therefore we can assume that $Y$ is an 
antichain. We aim at representing $Z := [Y]$ as an intersection of 
concatenations of stable upper sets to which the induction hypothesis 
applies.

Consider the set $K$ of front letters (i.e., prefixes of 
length $1$) of the words in $Y$. For $\delta \in K$ and 
any $W \subseteq \Lambda^{\ast}$ let $W_{\delta}$ be the set of 
words obtained from $W$ by cancelling all front letters $\delta$, 
that is, $x \in W_{\delta}$ if and only if either $\delta x \in W$, 
or $x \in W$ and $\delta$ is not a prefix of $x$. In case that $W$ 
is an upper set we simply have $W_{\delta} = \{x \in \Lambda^{\ast} 
: \delta x \in W\}.$ It is easy to see (by putting $\delta$ 
in front of all words in question) that each $W_{\delta}$ is a 
stable upper set whenever $W$ is a stable upper set. In particular, 
$Z_{\delta}$ is a stable upper set containing $Y_{\delta}$ (for 
$\delta \in K$). Therefore, as $Z$ is an upper set and 
$\uparrow Y \subseteq\;  \uparrow Y_{\delta}$ holds, we obtain the 
following inclusions
$$(\uparrow \delta) Z_{\delta} \subseteq Z = [Y] \subseteq 
[Y_{\delta}] \subseteq Z_{\delta}\;  \text{for all}\; \delta \in 
K.$$

Furthermore, $\parallel \!\!Y_{\delta}\!\! \parallel < \parallel \!\!Y 
\!\!\parallel$, and consequently $[Y_{\delta}]$ is closed by virtue of 
the induction hypothesis. These facts will be used in each of the 
subsequent cases (without explicit mention).\\

\noindent{\bf Case 1.} $K$ is a singleton $\{\delta\}.$\\This 
means that all words in $Y$ have the letter $\delta$ in front, that 
is, $Y = \delta Y_{\delta}$. Since the concatenation of stable 
upper sets is a stable upper set by Lemma \ref{lem:concatenation}, it follows
$$(\uparrow \delta) Z_{\delta} \subseteq Z = [\delta 
Y_{\delta}] \subseteq (\uparrow \delta) [Y_{\delta}] 
\subseteq (\uparrow \delta) Z_{\delta},$$
and thus equality holds throughout. Then $Z = (\uparrow \delta)
[Y_{\delta}]$ is closed by Lemma 2.1.\\

\noindent{\bf Case 2. } $K$ is not bounded below in $\Lambda.$\\
Let $\lambda$ be the meet of a maximal subset of $K$. Then 
there is a letter $\mu$ in $K$ incompatible with $\lambda$. 
Now if $x \in \displaystyle \bigcap_{\delta \in K} 
Z_{\delta}$, then $\delta x \in Z$ for all $\delta \in 
K.$ Applying the meet rule several times, we eventually get 
$\lambda x \in Z.$ Since $\mu x \in Z$, the cancellation rule returns 
$x \in Z$, thus proving that $Z$ contains the intersection of all $Z_{\delta}$ ($\delta\in K$). On the other hand, we already know that
$$Z \subseteq \displaystyle \bigcap_{\delta \in K} 
[Y_{\delta}] \subseteq \displaystyle \bigcap_{\delta \in K} 
Z_{\delta}.$$
Therefore $Z$ equals the intersection of all $[Y_{\delta}](\delta
\in K)$ and hence is closed.\\

\noindent{\bf Case 3. } $K$ is not a singleton, but bounded 
below.\\
Then  the meet $\alpha$ of $K$ in $\Lambda$ exists, and the 
hypothesis on $\Lambda$ guarantees an upper bound $\beta$ of 
$K$. Necessarily, $\alpha < \beta.$ Remove the front letters 
from all words in $Y$, which results in the set
$$Y_{K} = \{x \in \Lambda^{\ast} : \delta x \in Y\; \text{for some}\; \delta \in K\}=\bigcup_{\delta \in K}Y_{\delta}.$$
Since $\parallel \!\!Y_{\delta}\!\! \parallel < \parallel \!\!Y \!\!\parallel$, 
the stable set $[Y_{K}]$ must be closed. Note that
$$Y \subseteq (\uparrow \alpha) Y_{K}\ {\rm and}\ 
Y_{K} \subseteq Z_{\beta}$$
as $\alpha \leq \delta \leq \beta$ for all $\delta \in 
K.$ Hence, by Lemma \ref{lem:concatenation},
$$Z \subseteq [(\uparrow \alpha) Y_{K}] \subseteq (\uparrow 
\alpha) [Y_{K}] \subseteq (\uparrow \alpha) Z_{\beta}.$$
Now, applying the meet rule successively, we get $\alpha x \in Z$ 
whenever $\delta x \in Z$ for all $\delta \in K$ (where 
$x \in \Lambda^{\ast}$). Therefore
$$Z \subseteq \displaystyle \bigcap_{\delta \in \Lambda} 
[Y_{\delta}] \subseteq \displaystyle \bigcap_{\delta \in \Lambda} 
Z_{\delta} \subseteq Z_{\alpha}.$$
Combining both chains of inclusions yields
$$Z \subseteq \displaystyle \bigcap_{\delta \in K} 
[Y_{\delta}] \cap ((\uparrow \alpha) [Y_{K}]) \subseteq 
Z_{\alpha} \cap ((\uparrow \alpha) Z_{\beta}).$$

To prove the converse inclusion, assume $x \in Z_{\alpha} \cap 
((\uparrow \alpha) Z_{\beta}).$ Then $\alpha x \in Z$ and $x = wy$ for 
some word $w \geq \alpha$ and $y \in Z_{\beta}$. If $w\geq \beta$, then $x\in Z$ follows immediately. So, let $\beta\not \leq w$.  Writing $w = u \gamma 
v$ with $u, v \in \Lambda^{\ast}$ and a letter $\gamma \geq \alpha$, 
we have $\alpha u \gamma v y \in Z$ and $\beta y \in Z.$ Now  we can apply the 
compound rule (according to Lemma \ref{lem:compound}) and thus obtain $u'(\beta\wedge \gamma )v'y\in Z$, where $u'$ and $v'$ are some words below $u$ resp. $v$ (and only comprising letters below $\beta$) and $\alpha $ got 
removed because $\alpha \leq \beta \wedge \gamma$. Hence  $x = u \gamma v y \in Z$, as required. Hence $Z$ is equal to 
the intersection of all $[Y_{\delta}]$ with $(\uparrow 
\alpha)[Y_{K}].$ Since the latter set is closed by Lemma \ref{lem:concatenation}, 
so is $Z$. 

This completes the induction and thus establishes the 
equality of $[Y]$ and $Y^{\nabla \Delta}$ for all finite nonempty 
sets of words. These sets $Y^{\nabla \Delta}$ exhaust all closed 
upper sets when $\Lambda$ is well-founded. Indeed, in that case, the MacNeille completion of $\Lambda^*$ is well-founded  by Lemma \ref{lem2.4}.\end{proof}\\

Some of the four rules for generating the smallest stable upper set 
may become redundant in the case  of  particular alphabets. 
Specifically, we have the following consequence of Theorem \ref{theo:main}.

\begin{cor}
Let $\Lambda$ be a well-founded ordered set in which every (finite) 
subset bounded below has its meet and join in $\Lambda$. Then 
\begin{enumerate}[{(a)}]
\item  $\Lambda$ is a  lattice, 

 \item $\Lambda$ is a dual forest, 

\item  $\Lambda$ is a disjoint union of chains, 

 \item $\Lambda$ is a chain, or

\item  $\Lambda$ is an antichain, respectively, 
\end{enumerate}
if and only if the closed upper sets of $\Lambda^{\ast}$ 
are exactly the upper sets obeying the 
\begin{enumerate}[{(a)}]
\item  reduction, permutation, and meet rules,

 \item cancellation, reduction, and permutation rules,

\item  cancellation and reduction rules,

\item  reduction rule, or 

 \item cancellation rule, 
respectively.
\end{enumerate}
\end{cor}
\begin{proof} Necessity is clear. As to sufficiency, consider sets of the form

\noindent $(a)$ $\uparrow \{\alpha, \beta\}$ where $\alpha, \beta$ are 
incompatible, 

\noindent $(b)$ $\uparrow \{\alpha, \beta\}$ where $\alpha, \beta$ are 
incomparable, but bounded, 
 
\noindent $(c)$ $\uparrow \{\alpha \vee \beta, \alpha \beta\}$ where 
$\alpha, \beta$ are incomparable, 

\noindent $(e)$ $\uparrow \{\alpha \alpha, \beta\}$ for $\alpha< \beta$, respectively.

\noindent  In each case, the upper set as described is not closed, 
but obeys the corresponding subset of rules. Finally, (d) follows from 
(a) and (c). \end{proof}

\begin{conjecture} Let $\Lambda$ be a well-founded 
conditional lattice. Then an upper set $Z$ of $\Lambda^{\ast}$ is 
closed if and only if it satisfies the four rules.

\end{conjecture}
We do not even have a proof of this assertion in the simplest case of 
a $3$-letter alphabet $\Lambda = \{\lambda, \mu, \nu\}$ with $\nu < 
\lambda$ and $\nu < \mu$ ($\lambda, \mu$ being incomparable) so 
that $\nu = \lambda \wedge \mu.$

Theorem \ref{theo:main}   does not apply to the
ordered alphabets displayed in Figure 2. So,  we have not yet achieved a thorough understanding 
of the MacNeille completion of all free ordered monoids. It would also be interesting to characterize the closed upper sets 
of $\Lambda^{\ast}$ without imposing any condition on $\Lambda$;  then, 
of course, finitary rules are no longer sufficient:

\begin{problem} Characterize the closed upper sets of $\Lambda^{\ast}$ for an 
arbitrary ordered set $\Lambda.$
\end{problem}

\section{Final remark}
Originally, the main motivation for the description of upper sets belonging to the MacNeille  completion by means of syntactic rules was  to characterize absolute retracts among oriented graphs. The difficulty of a  characterization is due to  the fact that, in general, not every oriented graph (e.g., an oriented cycle) is isometrically embeddable in an absolute retract  in the category of oriented graphs (that is, a graph which is a  retract  of all its  isometric extensions). Our main result, Theorem \ref{theo:main},  entails   that  on a two-letter alphabet  $\Lambda:=\{+,-\}$, closed sets are characterized by the satisfaction of the cancellation rule.  This allows  to characterize among oriented graphs those which are  absolute retracts  in the category of oriented graphs. Indeed, it turns out that these  graphs are simply the retracts of products of oriented zigzags. This result, as well as others in the same vein, obtained  in collaboration with F.~Sa\"idane,  will be developped in a forthcoming paper. 

\section*{Acknowledgement} The second author thanks the members of the Mathematical Department of the University of Hamburg for their hospitality in February 2016, and the first author is grateful to the Camille Jordan  Institute for a visit in August 2017.

The authors are pleased to  thank the referee for his very careful reading.

\end{document}